\documentclass[11pt]{article}
\usepackage{amsmath, amsfonts, amssymb,amsthm,amscd}
\usepackage{graphicx}

\usepackage{amsmath,amsthm, amsxtra,amssymb,latexsym, amscd}
\usepackage[latin1]{inputenc}
\usepackage[mathscr]{eucal}
\newtheorem{theorem}{Theorem}[section]
\newtheorem*{theorem-a}{Theorem}

\newtheorem{lemma}[theorem]{Lemma}
\newtheorem{proposition}[theorem]{Proposition}
\newtheorem{example}[theorem]{Example}

\newtheorem{definition}[theorem]{Definition}
\newtheorem{remark}[theorem]{Remark}

\newcommand{\Acal}{{\mathcal A}}

\def\fz{\mathbb{Z}}
\def\fn{\mathbb{N}}
\def\fc{\mathbb{C}}

\DeclareMathOperator{\reg}{reg}

\begin{document}
\large
\centerline{\Large {\bf Hilbert series of Segre transform, }}
\centerline{\Large {\bf  and  Castelnuovo-Mumford regularity}}
\medskip
\vskip 0.7cm
\begin{center}

{{\sc Marcel
Morales}\\
{\small Universit\'e de Grenoble I, Institut Fourier, 
UMR 5582, B.P.74,\\
38402 Saint-Martin D'H\`eres Cedex,\\
and IUFM de Lyon, 5 rue Anselme,\\ 69317 Lyon Cedex (FRANCE)}\\
 }
 {{\sc Nguyen Thi Dung}\\
{\small Thai Nguyen University of Agriculture and Forestry, \\
Thai Nguyen,\\ Vietnam,\\ }}

\end{center}
\vskip 0.7cm

\noindent{\bf Abstract}{\footnote{ \it{ Partially supported by VIASM, Hanoi, Vietnam. }}} In a recent preprint, Ilse Fischer  and Martina Kubitzke,  proved the bilinearity  of the Segre transform 
 under some restricted hypothesis, motivated by their results  we show in this paper  the
 bilinearity  of the Segre transform in general. We apply these results to compute the postulation number of a series. 
Our second application is motivated by the paper of David A. Cox, and Evgeny Materov (2009), where is computed the Castelnuovo-Mumford regularity of the Segre Veronese embedding,
we can extend partially their result and compute the  Castelnuovo-Mumford regularity of the Segre product of Cohen-Macaulay modules.

{ \it{Key words and phrases: Segre-Veronese, Castelnuovo-Mumford regularity, Cohen-Macaulay, postulation number.  }}  
 \section{Introduction}
 In this paper, we deal only with formal Laurent series
$$ \mathfrak a=\sum_{l\geqslant \sigma_a}a_lt^l, \sigma_\mathfrak a\in \fz, a_l\in \fc $$
such that 
$$ (*)\ \ \  \mathfrak a=\frac{h(\mathfrak a)(t)}{(1-t)^{d_\mathfrak a}}, \text{ \ for some } d_\mathfrak a\geqslant 0, h(\mathfrak a)(t)\in \fc[t,t^{-1}]. $$
Given  two formal Laurent series $\mathfrak a, \mathfrak b$   satisfying (*), the Segre transform $\mathfrak a{\underline{\otimes}} \mathfrak b$ is defined by
$$\mathfrak a{\underline{\otimes}} \mathfrak b=\sum_{l\geqslant \sigma }a_lb_lt^l,  $$ 
where $\sigma=\max\{\sigma_\mathfrak a, \sigma_\mathfrak b\}.$
In a recent preprint \cite{fk}, the authors proved the bilinearity  of the Segre transform  under some restricted hypothesis, motivated by this results  we show in this paper  the
 bilinearity  of the Segre transform in general. We apply these results to compute the postulation number of a Segre product of series satisfying (*).
 Property (*)
 is equivalent to the existence of a polynomial $\Phi_{\mathfrak a} (t)\in \fc[t]$, such that $a_n=\Phi_{\mathfrak a} (n)$ for $n$ large enough. The postulation number is the smallest integer
 $\beta_{\mathfrak a} $ such that 
 $a_n=\Phi (n)$ for $n>\beta_{\mathfrak a} .$ It is well known that $\beta_{\mathfrak a} = \deg h(\mathfrak a)(t)-d_{\mathfrak a}.$
 
Our second application is motivated by the paper \cite{cm}, where is computed the Castelnuovo-Mumford regularity of the Segre Veronese embedding, we can extend partially their result.
 Our main result is:

\begin{theorem-a}  Let $S_1,\ldots, S_s$ be graded polynomial rings on disjoints of set of variables. 
For all $i=1,\ldots,s,$ let $M_i$ be a graded  finitely generated $S_i$-Cohen-Macaulay module. We assume that $M_i=\oplus_{l\geqslant 0}M_{i,l} $ as $S_i$-module. 
Let $d_i=\dim M_i, b_i=d_i-1\geq 0$, $\alpha_i=d_i-\reg(M_i),$ where $\reg(M_i)$ is the Castelnuovo-Mumford regularity of $M_i.$ If $\reg(M_i)<d_i,$ for all $i=1,\ldots,s$ then 

$(1)$ $M_1{\underline{\otimes}} \ldots {\underline{\otimes}} M_s$ is a Cohen-Macaulay $S_1{\underline{\otimes}} \ldots {\underline{\otimes}} S_s$-module.

$(2)$ $\reg(M_1{\underline{\otimes}} \ldots {\underline{\otimes}} M_s)=(b_1+\ldots+b_s+1)-\max\{\alpha_1,\ldots\alpha_s\}.$

$(3)$ For $n_i\in \fn,$ let $M_i^{<n_i>}$ be the $n_i$-Veronese transform of $M_i$, then 
$$\reg (M_1^{<n_1>}{\underline{\otimes}} \ldots {\underline{\otimes}} M_s^{<n_s>})=(b_1+\ldots+b_s+1)-\max\{\lceil\frac{\alpha_1}{n_1}\rceil,\ldots,\lceil\frac{\alpha_s}{n_s}\rceil\}.$$
\end{theorem-a} 
Note that this result can be proved easily by using local cohomology, but our purpose is to give a very elementary proof.

\centerline{\Large {\bf Segre transform of Laurent series}}
 
 In this paper, we deal only with formal Laurent series
$$ \mathfrak a=\sum_{l\geqslant \sigma_a}a_lt^l, \sigma_\mathfrak a\in \fz, a_l\in \fc $$
such that 
$$ (*)\ \ \  \mathfrak a=\frac{h(\mathfrak a)(t)}{(1-t)^{d_\mathfrak a}}, \text{ \ for some } d_\mathfrak a\geqslant 0, h(\mathfrak a)(t)\in \fc[t,t^{-1}]. $$
We will set $h(\mathfrak a)(t)=\sum_{n\geqslant \sigma_\mathfrak a}h_n(\mathfrak a)t^n.$

\begin{definition} Let $\mathfrak a, \mathfrak b$ be two formal Laurent series satisfying (*). the Segre transform $\mathfrak a{\underline{\otimes}} \mathfrak b$ is defined by
$$\mathfrak a{\underline{\otimes}} \mathfrak b=\sum_{l\geqslant \sigma }a_lb_lt^l,  $$ 
where $\sigma=\max\{\sigma_\mathfrak a, \sigma_\mathfrak b\}.$
\end{definition}
In all this paper we assume that $\mathfrak a{\underline{\otimes}} \mathfrak b\not=0$.
\begin{lemma} $\mathfrak a{\underline{\otimes}} \mathfrak b$ satisfies $(*).$ 
\end{lemma}
\begin{proof}
By \cite{Mo-Seg}, property (*) is equivalent to the existence of a polynomial $\Phi_\mathfrak a(l)$ such that
$ \Phi_\mathfrak a(l)=a_l $ for $l$ large enough. Moreover,
$$ \begin{cases}  d_\mathfrak a=\deg\Phi_\mathfrak a+1 \text{\ if }\Phi_\mathfrak a \text{\ is a non zero polynomial}\\
 d_\mathfrak a=0 \text{\ \ if\ } \Phi_\mathfrak a=0.\end{cases}$$

We have also a polynomial $\Phi_\mathfrak b(l)$ such that $\Phi_\mathfrak b(l)=b_l$ for $l$ large enough. Hence $a_lb_l=\Phi_\mathfrak a(l)\Phi_\mathfrak b(l)$ 
is a polynomial for $l$ large enough, 
and again by \cite{Mo-Seg}, there exist a Laurent polynomial $h(\mathfrak a{\underline{\otimes}} \mathfrak b)(t)$ such that
$$ \mathfrak a{\underline{\otimes}} \mathfrak b=\frac{h(\mathfrak a{\underline{\otimes}} \mathfrak  b)(t)}{(1-t)^{d_{\mathfrak a{\underline{\otimes}} \mathfrak b}}}, $$
where 
$$d_{\mathfrak a{\underline{\otimes}} \mathfrak b}= \begin{cases} 0 \text{\ if either \ } d_\mathfrak a=0 \text{\ or\ } d_\mathfrak b=0\\
 d_\mathfrak a+d_\mathfrak b-1 \text{\ if\ } d_\mathfrak a, d_\mathfrak b\geqslant 1.\end{cases}$$
\end{proof}
\begin{remark} We recall that binomial coefficients can be defined in a more general setting than natural numbers, indeed for  $k\in \fn,$ 
binomial coefficients are polynomial functions in the variable $n$. More precisely:

$(1)$ If $k=0$ then let ${n}\choose {0}$$=1$, for all $n\in \fc.$

$(2)$ If $k>0$ then let ${n}\choose{k}$=$\frac{n(n-1)\ldots(n-k+1)}{k!}$, for all $n\in \fc.$

Note that for all $n\in \fc$, ${{n}\choose {k}}=(-1)^k{{k-n-1}\choose{k}}$ and if $n\in \fn, n<k$, then ${n}\choose {k}$ $=0$. 

\end{remark}

\section{Segre transform is bilinear}

Let recall the following Lemma $1$ from \cite{Mo-Seg}.

\begin{lemma}\label{l1} Let 
$$ \mathfrak a=\sum_{l\geqslant \sigma_\mathfrak a}a_lt^l=\frac{h(\mathfrak a)(t)}{(1-t)^{d_\mathfrak a}}, $$
with $h(\mathfrak a)(t)=h_{\sigma_\mathfrak a}t^{\sigma_\mathfrak a}+\ldots +h_{r_\mathfrak a}t^{r_\mathfrak a}, d_\mathfrak a\geqslant 0$ and 
$ \sigma_\mathfrak a\leq r_\mathfrak a\in \fz.$
 We will set $b_\mathfrak a=d_\mathfrak a-1.$ Then for all $n=\sigma_\mathfrak a,\ldots, r_\mathfrak a$ we have 
$$ h_n(\mathfrak a)=\sum_{k=0}^{n-\sigma_\mathfrak a}(-1)^k{ {d_\mathfrak a}\choose {k}}a_{n-k}=\sum_{k=\sigma_\mathfrak a}^n(-1)^{n-k}{ {d_\mathfrak a}\choose{n-k}}a_k. \ \ \ (1) $$
On the other hand we have for all $k\geqslant \sigma_\mathfrak a$
$$ a_k=\sum_{i=0}^{k-\sigma_\mathfrak a}h_{k-i}(\mathfrak a){ {b_\mathfrak a+i}\choose{i}}=\sum_{i=\sigma_\mathfrak a}^kh_i(\mathfrak a){ {b_\mathfrak a+k-i}\choose{k-i}}\ \ (2) $$
\end{lemma} 
The first claim follows from the equality: $$ (\sum_{l\geqslant \sigma_\mathfrak a}a_lt^l)({(1-t)^{d_\mathfrak a}})={h(\mathfrak a)(t)}. $$
The second claim since:
$$ \sum_{l\geqslant \sigma_\mathfrak a}a_lt^l=\frac{h(\mathfrak a)(t)}{(1-t)^{d_\mathfrak a}} = (h(\mathfrak a)(t))(\sum_{i\geqslant 0}{ {b_\mathfrak a+i}\choose {i}}t^i)$$
The following two theorems extend \cite[Theorem 1]{fk}.
\begin{theorem}\label{t1} Let $\mathfrak a, \mathfrak b $ be two formal power series satisfying property (*). If $d_\mathfrak a=0$ , then for all $n\in \fz$
$$ h_n(\mathfrak a{\underline{\otimes}} \mathfrak b)=a_nb_n=h_n(\mathfrak a)\sum_{j=\sigma_\mathfrak b}^{n}h_j(\mathfrak b){ {b_\mathfrak b+n-j}\choose{n-j}}. $$
Moreover, $\deg h(\mathfrak a{\underline{\otimes}} \mathfrak b)(t)\leqslant \deg \mathfrak a$.  If $d_\mathfrak b>0$ and $h_j(\mathfrak b)\geqslant 0,$ for all $j$ then 
$\deg h(\mathfrak a{\underline{\otimes}} \mathfrak b)(t)=\deg(\mathfrak a).$
\end{theorem}

\begin{proof} Since $d_\mathfrak a=0$,  $\mathfrak a$ is a Laurent polynomial, we have that $ h_n(\mathfrak a)=a_n$, for all $n\in \fz$ and $a_n=0 $ for 
$n> \deg(\mathfrak a),$ which implies that
$$  \mathfrak a{\underline{\otimes}} \mathfrak b=\sum_{n\leqslant \deg(\mathfrak a)}a_nb_nt^n.$$
hence $\deg h(\mathfrak a{\underline{\otimes}} \mathfrak b)(t)\leq \deg(\mathfrak a).$ 
Now suppose that $d_\mathfrak b>0$ and $h_j(\mathfrak b)\geqslant 0$ for all $j$, by using equation (2) we have,
$$ h_n(\mathfrak a{\underline{\otimes}} \mathfrak b)=h_n(\mathfrak a)\sum_{j=\sigma_\mathfrak b}^{n}h_j(\mathfrak b){ {b_\mathfrak b+n-j}\choose{n-j}}. $$
Note that  $h_{\sigma_\mathfrak b}>0$, ${ {b_\mathfrak b+n-\sigma_\mathfrak b}\choose{n-\sigma_\mathfrak b}}>0$, for $n\geq \sigma_\mathfrak b$.
 The assumption $(\mathfrak a{\underline{\otimes}} \mathfrak b)\ne 0$ implies $\sigma_\mathfrak b\leqslant \deg \mathfrak a,$ so
 $\sum_{j=\sigma_\mathfrak b}^{\deg \mathfrak a}h_j(\mathfrak b){ {b_\mathfrak b+{\deg \mathfrak a}-j}\choose {{\deg \mathfrak a}-j}}>0.$ 
Hence $h_{\deg \mathfrak a}(\mathfrak a{\underline{\otimes}} \mathfrak b)>0$ and the claim is over.
\end{proof}
\begin{remark} Let $\mathfrak a, \mathfrak b $ be two formal power Laurent series satisfying property (*), the product $h_i(\mathfrak a)h_j(\mathfrak b)$ 
is null for any $i<\sigma ,j<\sigma $, where $\sigma $ is any 
of the numbers $\sigma_{\mathfrak a}, \sigma_{\mathfrak b}, \sigma_{(\mathfrak a{\underline{\otimes}} \mathfrak b)}=\max(\sigma_{\mathfrak a}, \sigma_{\mathfrak b}),\min(\sigma_{\mathfrak a}, \sigma_{\mathfrak b}). $

\end{remark} 
\begin{theorem}\label{t2} Let $\mathfrak a, \mathfrak b $ be any formal power Laurent series satisfying property (*).
 Let $b_{\mathfrak a}=d_{\mathfrak a}-1\geqslant 0, b_{\mathfrak b}=d_{\mathfrak b}-1\geqslant 0$ and $\sigma $ be any 
of the numbers $\sigma_{\mathfrak a}, \sigma_{\mathfrak b}, \sigma_{(\mathfrak a{\underline{\otimes}} \mathfrak b)}=\max(\sigma_{\mathfrak a}, \sigma_{\mathfrak b}),\min(\sigma_{\mathfrak a}, \sigma_{\mathfrak b}). $
 Then for any $n\in \fz$ 
$$ h_n(\mathfrak a{\underline{\otimes}} \mathfrak b)=\sum_{i=\sigma}^{\infty}\sum_{j=\sigma}^{\infty}h_i(\mathfrak a)h_j(\mathfrak b){ {b_{\mathfrak a}+j-i}\choose{n-i}}{ {b_{\mathfrak b}+i-j}\choose {n-j}}. $$ 
\end{theorem}

The proof will follow immediately from the  Lemma \ref{l1} and the proof of \cite[Theorem 1]{fk}:

\section{Postulation number, Castelnuovo-Mummford regularity}

\begin{lemma} \label{l6} Let $i,j\in \fz, b_1:=d_1-1\geqslant 0, b_2:=d_2-1\geqslant 0$ and  
$$ T_{d_1}^i:=\frac{t^i}{(1-t)^{d_1}};  T_{d_2}^j:=\frac{t^j}{(1-t)^{d_2}} $$
Then 
$$ T_{d_1}^i{\underline{\otimes}} T_{d_2}^j =\frac{\sum_{n=\max(i,j)}^{r_{i,j}} { {b_1+j-i}\choose {n-i}} { {b_2+i-j}\choose {n-j}}t^n}{(1-t)^{b_1+b_2+1}},\ \ (3)$$
where :
$$ r_{i,j}=\begin{cases} \min (b_1+j, b_2+i) \text{\ if\ } b_2+i-j\geqslant 0 \text{ \ and\ } b_1+j-i\geqslant 0\\
 \max (b_1+j, b_2+i) \text{\ if\ } b_2+i-j<0 \text{ \ or\ } b_1+j-i< 0. \end{cases} $$
\end{lemma}

\begin{proof} The equality (3) follows from theorem \ref{t2}. We need  to check that for $n>r_{i,j}$, we have ${{b_1+j-i}\choose {n-i}} { {b_2+i-j}\choose {n-j}}=0$, and that for  
$n=r_{i,j}$, we have ${{b_1+j-i}\choose {n-i}} { {b_2+i-j}\choose {n-j}}\not=0.$  We have two cases:
\begin{enumerate}
\item  If $b_1+j-i\geq 0$ and $b_2+i-j\geq 0$, then ${{b_1+j-i}\choose {n-i}}=0$ 
if and only if $b_1+j<n$. Hence $r_{i,j}= \min (b_1+j, b_2+i). $
\item  Either $b_1+j-i< 0$ or $b_2+i-j< 0$. Suppose for example that $b_1+j-i< 0$ then ${{b_1+j-i}\choose {n-i}}\not=0$,
 and  ${ {b_2+i-j}\choose {n-j}}=0$
if and only if $b_2+i<n$, but $b_1+j<i\leq b_2+i<n$. Hence $r_{i,j}= \max (b_1+j, b_2+i). $
\end{enumerate}
\end{proof}
\begin{example}\label{e1}   Let $\alpha, \beta \in \fz$, we study $T_d^{\alpha}{\underline{\otimes}} T_1^{\beta}$.
We consider two cases:

$(1)$ If $\max(\alpha, \beta)=\alpha$ then
$$ T_d^{\alpha}{\underline{\otimes}} T_1^{\beta}=\frac{t^{\alpha}}{(1-t)^d}. $$

$(2)$ If $\max(\alpha, \beta)=\beta>\alpha$ then

 $$T_d^{\alpha}{\underline{\otimes}} T_1^{\beta}
=\frac{t^{\alpha}-t^{\alpha}(\sum_{l=0}^{\beta-\alpha-1} {{d-1+l}\choose{l}}t^l)(1-t)^d}{(1-t)^d}.$$
Note that $\deg( t^{\alpha}-t^{\alpha}(\sum_{l=0}^{\beta-\alpha-1} {{d-1+l}\choose{l}}t^l)(1-t)^d)=d-1+\beta $.
\end{example}
\begin{proposition} \label{p7} Let 
$$\mathfrak a=\frac{h(\mathfrak a)(t)}{(1-t)^{d_{\mathfrak a}}}; \mathfrak b=\frac{h(\mathfrak b)(t)}{(1-t)^{d_{\mathfrak b}}}, \text{ \ where\ }
h(\mathfrak a)(t), h(\mathfrak b)(t) \in \fc[t,t^{-1}]. $$
For any non null Laurent series satisfying property (*), we denote $\sigma _{\mathfrak a}=\min_n  h_n(\mathfrak a)\not=0, r_{\mathfrak a}=\deg h(\mathfrak a)(t)$. Then

(1) $ r_{\mathfrak a{\underline{\otimes}} \mathfrak b}\leqslant \max(b_\mathfrak a+r_{\mathfrak b}, b_{\mathfrak b}+r_{\mathfrak a}).$ 

(2) If for all $\sigma_{\mathfrak a}\leqslant i\leqslant r_{\mathfrak a}, \sigma_{\mathfrak b}\leqslant j\leqslant r_{\mathfrak b}$ such that $h_i(\mathfrak a)\ne 0, h_j(\mathfrak b)\ne 0$ we have
$ b_{\mathfrak b}+i-j\geqslant 0 \text {\ and \ } b_{\mathfrak a}+j-i\geqslant 0 $ then 
$$ r_{\mathfrak a{\underline{\otimes}} \mathfrak b}\leqslant \min(b_\mathfrak a+r_{\mathfrak b}, b_{\mathfrak b}+r_{\mathfrak a}).  $$
Moreover, if for all $i,j, h_i(\mathfrak a)\geqslant 0, h_j(\mathfrak b)\geqslant 0$ then $r_{\mathfrak a{\underline{\otimes}} \mathfrak b}=\min(b_\mathfrak a+r_{\mathfrak b}, b_{\mathfrak b}+r_{\mathfrak a}).$

(3) If $0\leqslant \sigma_{\mathfrak a}\leqslant r_{\mathfrak a}\leqslant b_{\mathfrak a}$ and $0\leqslant \sigma_{\mathfrak b}\leqslant r_{\mathfrak b}\leqslant b_{\mathfrak b}$, then 
$$  r_{\mathfrak a{\underline{\otimes}} \mathfrak b}\leqslant \min(b_\mathfrak a+r_{\mathfrak b}, b_{\mathfrak b}+r_{\mathfrak a}).$$ 
Moreover, if for all $i,j, h_i(\mathfrak a)\geqslant 0, h_j(\mathfrak b)\geqslant 0$ then $r_{\mathfrak a{\underline{\otimes}} \mathfrak b}=\min(b_\mathfrak a+r_{\mathfrak b}, b_{\mathfrak b}+r_{\mathfrak a}).$
\end{proposition}
\begin{proof}
(1) Let $n>\max(b_\mathfrak a+r_{\mathfrak b}, b_{\mathfrak b}+r_{\mathfrak a})$, for all $i\leqslant r_{\mathfrak a}, j\leqslant r_{\mathfrak b}$. By the Lemma \ref{l6}, this implies that 
$A_{i,j,n}=0, $  for all $i\leqslant r_{\mathfrak a}, j\leqslant r_{\mathfrak b}$. Hence $h_n(\mathfrak a{\underline{\otimes}} \mathfrak b)=0,$ which implies that $\sigma_(\mathfrak a{\underline{\otimes}} \mathfrak b)\leqslant \max(b_\mathfrak a+j, b_{\mathfrak b}+i).$

(2) Since for all $\sigma_{\mathfrak a}\leqslant i\leqslant r_{\mathfrak a}, \sigma_{\mathfrak b}\leqslant j\leqslant r_{\mathfrak b}$, we have that $b_{\mathfrak b}+i-j\geqslant 0$ and 
$b_{\mathfrak a}+j-i\geqslant 0,$ by the Lemma \ref{l6}, we have that $r_{i,j}=\min(b_\mathfrak a+j, b_{\mathfrak b}+i).$ This implies for all $i,j$ that
$$ A_{i,j,r_{i,j}}\ne 0 \text{\ and\ } A_{i,j,n}=0 \text{\ for\ } n>r_{i,j}. $$
On the other hand
$$ \min(b_\mathfrak a+r_{\mathfrak b}, b_{\mathfrak b}+r_{\mathfrak a})\geqslant \min(b_\mathfrak a+r_{\mathfrak b}, b_{\mathfrak b}+i)\geqslant \min(b_\mathfrak a+j, b_{\mathfrak b}+i), $$
for all $\sigma_{\mathfrak a}\leqslant i\leqslant r_{\mathfrak a}, \sigma_{\mathfrak b}\leqslant j\leqslant r_{\mathfrak b}.$ 
Hence $A_{i,j,n}=0$ for $n>\min(b_\mathfrak a+r_{\mathfrak b}, b_{\mathfrak b}+r_{\mathfrak a}).$ Note that the conditions $b_\mathfrak a+j-i\geqslant 0,  b_{\mathfrak b}+i-j \geqslant 0$ 
implies that $A_{i,j,n}\geqslant 0$ for all $n$. Hence if for all $i,j,h_i(\mathfrak a)\geqslant 0, h_j(\mathfrak b)\geqslant 0$ then $h_n({\mathfrak a{\underline{\otimes}}\mathfrak b})\geqslant 0$ 
for all $n$,
and for $ m:=\min(b_\mathfrak a+r_{\mathfrak b}, b_{\mathfrak b}+r_{\mathfrak a})$ we have:
$$ h_{m}({\mathfrak a{\underline{\otimes}}\mathfrak b})=\underset{i,j\mid   A_{i,j,m}\ne 0}
\sum h_i(\mathfrak a)h_j(\mathfrak b)A_{i,j,m}>0. $$ 

(3) If $0\leqslant \sigma_{\mathfrak a}\leqslant r_{\mathfrak a}\leqslant b_{\mathfrak a}$ and $0\leqslant \sigma_{\mathfrak b}\leqslant r_{\mathfrak b}\leqslant b_{\mathfrak b}$ then $b_{\mathfrak a}+j\geqslant b_{\mathfrak a}\geqslant i$ and $b_{\mathfrak b}+i\geqslant b_{\mathfrak b}\geqslant j$. Therefore  $b_{\mathfrak a}+j-i\geqslant 0$ and  $b_{\mathfrak b}+i-j\geqslant 0$. Hence, the claim follows from the claim $2.$ 
\end{proof}

\begin{remark} The bounds obtained are sharp.
\end{remark}

\begin{lemma}\label{l8} The following statements are equivalent:

(1) For all $\sigma_{\mathfrak a}\leqslant i\leqslant r_{\mathfrak a}$ and $\sigma_{\mathfrak b}\leqslant j\leqslant r_{\mathfrak b}$, we have 
$$ b_{\mathfrak b}+i-j\geqslant 0 \text{\ and\ } b_{\mathfrak a}+j-i\geqslant 0.$$ 

(2) $b_{\mathfrak b}+\sigma_{\mathfrak a}-r_{\mathfrak b}\geqslant 0 \text{\ and\ } b_{\mathfrak a}+\sigma_{\mathfrak b}-r_{\mathfrak a}\geqslant 0. \ \ \ (**)$
\end{lemma}

\begin{proof} $(2)  \Rightarrow (1).$ Take $i=\sigma_{\mathfrak a}, j=r_{\mathfrak b}$ in the first inequality and  $j=\sigma_{\mathfrak b}, i=r_{\mathfrak a}$ in the second.

$(1)  \Rightarrow (2).$ Let $\sigma_{\mathfrak a}\leqslant i\leqslant r_{\mathfrak a}, \sigma_{\mathfrak b}\leqslant j\leqslant r_{\mathfrak b}$, then 
$$ i+b_{\mathfrak b}-j\geqslant i+b_{\mathfrak b}-r_{\mathfrak b}\geqslant \sigma_{\mathfrak a}+b_{\mathfrak b}-r_{\mathfrak b}\geqslant 0 \text{\ and \ }  j+b_{\mathfrak a}-i\geqslant i+b_{\mathfrak a}-r_{\mathfrak a}\geqslant \sigma_{\mathfrak b}+b_{\mathfrak a}-r_{\mathfrak a}\geqslant 0. $$
\end{proof}

\begin{remark} \label{r1}Suppose that $M_{\mathfrak a}, M_{\mathfrak b}$ are Cohen-Macaulay modules of dimensions $d_{\mathfrak a}=b_{\mathfrak a}+1\geq 2$,
 $d_{\mathfrak b}=b_{\mathfrak b}+1\geq 2$,
 with Hilbert-Poincaré series ${\mathfrak a}, {\mathfrak b}$. 
 Then the conditions
$b_{\mathfrak b}+\sigma_{\mathfrak a}-r_{\mathfrak b}\geqslant 0 $ and $b_{\mathfrak a}+\sigma_{\mathfrak b}-r_{\mathfrak a}\geqslant 0$ are equivalent to say that
 $M_{\mathfrak a}{\underline{\otimes}} M_{\mathfrak b}$ is a Cohen-Macaulay module by \cite{gw}[Proposition (4.2.5)].
\end{remark}

\begin{proposition} \label{p9} Let consider ${\mathfrak a}_1,\ldots, {\mathfrak a}_s$ Laurent formal series satisfying (*)
$${\mathfrak a}_i=\frac{h({\mathfrak a}_i)(t)}{(1-t)^{d_i}}, h({\mathfrak a}_i)(t)\in \fc[t,t^{-1}].  $$
We set $r_i=\deg h({\mathfrak a}_i)(t), \alpha_i=d_i-r_i, b_i=d_i-1\geqslant 0$ and 
$$ {\mathfrak a}_1{\underline{\otimes}} \ldots {\underline{\otimes}} {\mathfrak a}_s =\frac{h({\mathfrak a}_1{\underline{\otimes}} \ldots {\underline{\otimes}} {\mathfrak a}_s)(t)}{(1-t)^{b_1+\ldots +b_s+1}}.$$
Then 

(1) $ \deg(h({\mathfrak a}_1{\underline{\otimes}} \ldots {\underline{\otimes}} {\mathfrak a}_s))\leqslant (b_1+\ldots +b_s+1)-\min(\alpha_1,\ldots,\alpha_s).$

(2) If the condition (**) of Lemma \ref{l8} is fulfilled for 
$$ \{{\mathfrak a}_1,{\mathfrak a}_2\}, \{{\mathfrak a}_1{\underline{\otimes}} {\mathfrak a}_2,{\mathfrak a}_3\},\ldots, \{{\mathfrak a}_1{\underline{\otimes}} \ldots {\underline{\otimes}} {\mathfrak a}_{s-1},{\mathfrak a}_s\} $$
and $h_k({\mathfrak a}_i)\geqslant 0$ for all $i$ and $k$, then 
$$ \deg(h({\mathfrak a}_1{\underline{\otimes}} \ldots {\underline{\otimes}} {\mathfrak a}_s))=(b_1+\ldots +b_s+1)-\max(\alpha_1,\ldots,\alpha_s). $$

(3) If for all $i=1,...,s$, $0\leq \sigma_i  r_i<d_i,$  then the condition  (**) of Lemma \ref{l8} is fulfilled for 
$$ \{{\mathfrak a}_1,{\mathfrak a}_2\}, \{{\mathfrak a}_1{\underline{\otimes}} {\mathfrak a}_2,{\mathfrak a}_3\},\ldots, \{{\mathfrak a}_1{\underline{\otimes}} \ldots {\underline{\otimes}} {\mathfrak a}_{s-1},{\mathfrak a}_s\} .$$
\end{proposition}

\begin{proof} (1) Note that 
$$ \max(b_1+r_2, b_2+r_1)=\max(b_1+b_2+1-\alpha_2, b_1+b_2+1-\alpha_1)=b_1+b_2+1-\min(\alpha_1,\alpha_2). $$
Now suppose $s\geqslant 3$, we prove the claim by induction. Assume that (1) is true  for the case $s-1:$
 $$ \deg(h({\mathfrak a}_1{\underline{\otimes}} \ldots {\underline{\otimes}} {\mathfrak a}_{s-1}))\leqslant (b_1+\ldots +b_{s-1}+1)-\min(\alpha_1,\ldots,\alpha_{s-1}).$$
Then, by the Proposition \ref{p7}
$$ \begin{aligned} \deg(&h({\mathfrak a}_1{\underline{\otimes}} \ldots {\underline{\otimes}} {\mathfrak a}_{s-1}) {\underline{\otimes}} {\mathfrak a}_s))\leqslant \max(b_1+\ldots +b_{s-1}+r_s, b_s+\deg({\mathfrak a}_1{\underline{\otimes}} \ldots{\underline{\otimes}} {\mathfrak a}_{s-1})\\
& \leqslant  \max (b_1+\ldots +b_{s-1}+b_s+1-\alpha_s, b_1+\ldots +b_s+1-\min(\alpha_1,\ldots,\alpha_{s-1})\\
&=b_1+\ldots +b_s+1+\max(-\alpha_s,-\min(\alpha_1,\ldots, \alpha_{s-1}))\\
&=b_1+\ldots +b_s+1-\min(\alpha_s, \min(\alpha_1,\ldots, \alpha_{s-1}))   \\
&=b_1+\ldots +b_s+1-\min(\alpha_s, \min(\alpha_1,\ldots, \alpha_{s}))  \end{aligned} $$

(2) Since condition (**) is fulfilled, we can apply Proposition \ref{p7} and we get: 
$$ \begin{aligned} \deg({\mathfrak a}_1{\underline{\otimes}} {\mathfrak a}_2)&\leqslant \min(b_1+r_2, b_2+r_1)=\min(b_1+b_2+1-\alpha_2, b_2+b_1+1-\alpha_1)\\
&=b_1+b_2+1+\min(-\alpha_1,-\alpha_2)=b_1+b_2+1-\max(\alpha_1, \alpha_2)  \end{aligned} $$
and we have equality if $h_k({\mathfrak a}_i)\geqslant 0,$ for all $i,k.$ By induction hypothesis we assume that 

$$ \deg({\mathfrak a}_1{\underline{\otimes}} \ldots {\underline{\otimes}} {\mathfrak a}_{s-1})\leqslant (b_1+\ldots +b_{s-1}+1-\max(\alpha_1,\ldots, \alpha_{s-1}) \ \ (4) $$ and  we have equality 
if $h_k({\mathfrak a}_i)\geqslant 0$, for all $i,k$. Moreover by Proposition \ref{p7},the coefficients $h_k({\mathfrak a}_1{\underline{\otimes}} \ldots {\underline{\otimes}} {\mathfrak a}_{s-1})$  are $\geqslant 0.$
 On the other hand, condition (**) is fulfilled, so we can apply Proposition \ref{p7}, we have:

$$ \deg({\mathfrak a}_1{\underline{\otimes}} \ldots {\underline{\otimes}} {\mathfrak a}_{s-1}) {\underline{\otimes}} {\mathfrak a}_s))\leqslant \min(b_1+\ldots +b_{s-1}+r_s, b_s+\deg({\mathfrak a}_1{\underline{\otimes}} \ldots{\underline{\otimes}} {\mathfrak a}_{s-1}) \ \ (5), $$
where we have the equality if $h_k({\mathfrak a}_s)\geqslant 0$ and $h_k({\mathfrak a}_1{\underline{\otimes}} \ldots {\underline{\otimes}} {\mathfrak a}_{s-1})\geqslant 0,$ which is true since by hypothesis $h_k({\mathfrak a}_i)\geqslant 0$ for all $i,k$.

Using (4) in (5) we get
$$ \begin{aligned} \deg(h({\mathfrak a}_1{\underline{\otimes}}& \ldots {\underline{\otimes}} {\mathfrak a}_s))\leqslant\\ &\min(b_1+\ldots +b_{s-1}+1-\alpha_s, b_s+b_1+\ldots+b_{s-1}+1-\max(\alpha_1, \alpha_{s-1}))\\
&=b_1+\ldots +b_s+1+\min(-\alpha_s,-\max(\alpha_1,\ldots, \alpha_{s-1}))\\
&=b_1+\ldots +b_s+1-\max(\alpha_s, \max(\alpha_1,\ldots, \alpha_{s-1}))   \\
&=b_1+\ldots +b_s+1-\max(\alpha_s, \min(\alpha_1,\ldots, \alpha_{s}))  \end{aligned} $$
and we have the equality if $h_k({\mathfrak a}_i)\geqslant 0$ for all $i,k.$

(3) The proof is immediate from Proposition \ref{p7}.
\end{proof}
\section{$h-$ vector of the Segre product of $s$ power series}
The proof of the following theorem is direct from \ref{t2} by using induction.
\begin{theorem} \label{thl6} With the notations of Proposition \ref{p9}. 

For 
$\sigma_{s}\leq   i_{s}\leq  b_{1}+...+b_{s}-\min \{\alpha _{1},...,\alpha _{s}\}$ we have
$$h_{i_s}({\mathfrak a_1}{ {\underline{\otimes}} } ... { {\underline{\otimes}} }{\mathfrak a_s})= 
 \sum_{(i_1,i_2,...,i_{s-1 },l_2,...,l_s)\in \Delta }h_{i_{1 } }({\mathfrak a_1})h_{l_{2 } }({\mathfrak a_2})...h_{l_{s } }
({\mathfrak a_s})A_{i_1,l_2,i_2 }...A_{i_{s-1 },l_s,i_s }$$
where $$\forall k=2,...,s; \  A_{i_{k-1 },l_k,i_k }={ {b_{1}+...+b_{{k-1 }}+l_k-i_{k-1 }}\choose 
{i_k-i_{k-1 }}}{ {b_{k}+i_{k-1 }-l_k}\choose {i_k-l_k}},$$and  $\Delta $ is defined by 
: for any $\tau =2,...,s$, 
$$\sigma_\tau  \leq    l_{\tau} \leq b_{\tau}+1- \alpha _{\tau} ,\  \  
 \sigma_{\tau -1}\leq   i_{\tau-1}\leq \min \{ b_{1}+...+b_{\tau-1}-\min \{\alpha _{1},...,\alpha _{\tau-1}\}, i_{\tau}  \}.$$

\end{theorem}
There is an important corollary that will be used in \cite{Mo-Seg} to prove the conjecture by Simon Newcomb:
\begin{theorem} \label{thl7} For $j=1,...,n,$  let $S_j$ be a polynomial ring over a field $K$ in $b_j +1$ variables, ${\mathfrak a_j}$  the 
 Hilbert-Poincare series of $S_j$, that is   ${\mathfrak a_j}= \displaystyle \frac{1}{(1-t)^{b_j+1}}$ then: 
 
For $k =0,..., b_{1}+...+b_{n}-\max \{b_{1},...,b_{n}\}$, we have 
$$A({[{\bf b}]},k)=\sum_{(i_2,...,i_{n-1})\in \Delta } A_{i_2}  A_{i_2,i_3}A_{i_3,i_4}... A_{i_{n-1},i_{n}},$$
where $$i_{n}:=k; A_{i_2}={ {b_1 }\choose{i_2}}{ {b_2}\choose {i_2}}; \  \forall  s= 2,...,n-1,\  A_{i_s,i_{s+1}}={ {b_1+...+b_{s}-i_{s}}\choose{i_{s+1}-i_{s}}}{ {b_{s+1}+i_{s}}\choose { i_{s+1}}}\geq 0$$ 
 and $\Delta $ is defined by : for any $\tau =2,...,n-1$ $$0\leq    i_{\tau }\leq \min \{ b_{1}+...+b_{\tau }-\max \{b_{1},...,b_{\tau }\}, i_{\tau +1}  \}$$
\end{theorem}
Applying Proposition \ref{p9} to modules we get the following Theorem (note that this theorem can be proved easily by using \cite{gw},
 but our purpose is to prove it by using only elementary tools):
\begin{theorem} \label{main1} Let $S_1,\ldots, S_s$ be graded polynomial rings on disjoints  sets of variables. 
For all $i=1,\ldots,s,$ let $M_i$ be a graded  finitely generated $S_i$-Cohen-Macaulay module. We assume that $M_i=\oplus_{l\geqslant 0}M_{i,l} $ as $S_i$-module. 
Let $d_i=\dim M_i, b_i=d_i-1\geq 0$, $\alpha_i=d_i-\reg(M_i),$ where $\reg(M_i)$ is the Castelnuovo-Mumford regularity of $M_i.$ If $\reg(M_i)<d_i,$ for all $i=1,\ldots,s$ then 

$(1)$ $M_1{\underline{\otimes}} \ldots {\underline{\otimes}} M_s$ is a Cohen-Macaulay $S_1{\underline{\otimes}} \ldots {\underline{\otimes}} S_s$-module.

$(2)$ $\reg(M_1{\underline{\otimes}} \ldots {\underline{\otimes}} M_s)=(b_1+\ldots+b_s+1)-\max\{\alpha_1,\ldots\alpha_s\}.$

$(3)$ For $n_i\in \fn,$ let $M_i^{<n_i>}$ be the $n_i$-Veronese transform of $M_i$, then 
$$\reg (M_1^{<n_1>}{\underline{\otimes}} \ldots {\underline{\otimes}} M_s^{<n_s>})=(b_1+\ldots+b_s+1)-\max\{\lceil\frac{\alpha_1}{n_1}\rceil,\ldots,\lceil\frac{\alpha_s}{n_s}\rceil\}.$$
\end{theorem} 
\begin{proof}We have that for all $i=1,...,s$, $0\leq \sigma (M_i),  \reg(M_i)<d_i,$  then statement (3) of Proposition \ref{p9} implies that 
is a Cohen-Macaulay module by the Remark \ref{r1} and \cite{gw}[Proposition (4.2.5)]. The second claim follows immediately from  Proposition \ref{p9}.

The third claim follows from the second and the fact that for all $i=1,...,s$, $\reg (M_i^{<n_i>})= d_i-\lceil\frac{\alpha_i}{n_i}\rceil$, proved in\cite{md}[Theorem 4.7].
\end{proof}
If one of the modules has dimension 0, then we get the following Corollary of Theorem \ref{t1}.
\begin{theorem}  Let $S_1,\ldots, S_s$ be graded polynomial rings on disjoints of set of variables. 
For all $i=1,\ldots,s,$ let $M_i$ be a graded  finitely generated $S_i$-Cohen-Macaulay module of dimension $d_i=\dim M_i. $ We assume that $M_i=\oplus_{l\in \fz}M_{i,l} $ as $S_i$-module, 
and there is an index $k$ such that $d_k=\dim M_k=0$.
Then 
$M_1{\underline{\otimes}} \ldots {\underline{\otimes}} M_s$ is a 0-dimensional Cohen-Macaulay $S_1{\underline{\otimes}} \ldots {\underline{\otimes}} S_s$-module and 
$\reg(M_1{\underline{\otimes}} \ldots {\underline{\otimes}} M_s)=\underset{k\mid \dim M_k=0}{\min} \reg(M_k)$.
\end{theorem} 
To end this section we exhibit two large classes of ideals that satisfy the hypothesis of Theorem \ref{main1}.
\begin{description}
\item[(I)] Let $\fn \Acal$ be a finite generated normal semigroup homogeneous, then by \cite{st}[13.14], we have that $\reg(K[\fn \Acal])< \dim (K[\fn \Acal])$. Hence the toric ring 
$K[\fn \Acal]$ satisfy the hypothesis of Theorem \ref{main1}. 
\item[(II)] Let $\Delta $ be a simplicial complex, and $K[\Delta ]$ be the Stanley-Reisner ring associated to $\Delta $. If $K[\Delta ]$ is a Cohen-Macaulay ring,
then by the main theorem of Reisner $\reg(K[\Delta ])< \dim (K[\Delta ])$ if and only if $\Delta $ is acyclic.
\end{description}

\end{document}